\documentclass[12pt]{article}
\usepackage{amsmath,amsthm,amsfonts,amssymb,mathrsfs,latexsym,graphicx}
\newtheorem{thm}{Theorem}
\newtheorem{lem}{Lemma}
\newtheorem{prop}{Proposition}
\newtheorem{conj}{Conjecture}

\newcommand{\lrf}[1]{\left\lfloor #1\right\rfloor}
\newcommand{\lrc}[1]{\left\lceil #1\right\rceil}
\newcommand{\D}{d}
\renewcommand{\d}{\delta}

\title{On unimodality problems in Pascal's triangle
\thanks{Partially supported by the National Science Foundation of China under Grant No.10771027.}}
\author{Xun-Tuan Su and Yi Wang
\thanks{Corresponding author.}\\
\small Department of Applied Mathematics\\[-0.8ex]
\small Dalian University of Technology\\[-0.8ex]
\small Dalian 116024, P. R. China\\
\small \texttt{suxuntuan@yahoo.com.cn}\\[-0.8ex]
\small \texttt{wangyi@dlut.edu.cn}}

\date{\normalsize Submitted: Jan 23, 2008\quad Accepted: Aug 28, 2008\quad Published: Sep 8, 2008\\
Mathematics Subject Classification: 05A10, 05A20}

\begin{document}
\maketitle

\begin{abstract}
Many sequences of binomial coefficients share various unimodality
properties. In this paper we consider the unimodality problem of a
sequence of binomial coefficients located in a ray or a transversal
of the Pascal triangle. Our results give in particular an
affirmative answer to a conjecture of Belbachir {\it et al} which
asserts that such a sequence of binomial coefficients must be
unimodal. We also propose two more general conjectures.
\end{abstract}

\section{Introduction}
Let $a_0,a_1,a_2,\ldots$ be a sequence of nonnegative numbers. It is
called {\it unimodal} if $a_0\le a_1\le\cdots\le a_{m-1}\le a_m\ge
a_{m+1}\ge\cdots$ for some $m$ (such an integer $m$ is called a mode
of the sequence). In particular, a monotone (increasing or
decreasing) sequence is known as unimodal. The sequence is called
{\it concave} (resp. {\it convex}) if for $i\ge 1$,
$a_{i-1}+a_{i+1}\le 2a_i$ (resp. $a_{i-1}+a_{i+1}\ge 2a_i$). The
sequence is called {\it log-concave} (resp. {\it log-convex}) if for
all $i\ge 1$, $a_{i-1}a_{i+1}\le a_i^2$ (resp. $a_{i-1}a_{i+1}\ge
a_i^2$). By the arithmetic-geometric mean inequality, the concavity
implies the log-concavity (the log-convexity implies the convexity).
For a sequence $\{a_i\}$ of positive numbers, it is log-concave
(resp. log-convex) if and only if the sequence
$\left\{a_{i+1}/a_i\right\}$ is decreasing (resp. increasing), and
so the log-concavity implies the unimodality. The unimodality
problems, including concavity (convexity) and log-concavity
(log-convexity), arise naturally in many branches of mathematics.
For details, see
\cite{Bre89,Bre94,Sta89,Wjcta02,Weujc02,Wlaa03,WYeujc05,WYjcta07}
about the unimodality and log-concavity and \cite{DV08,LW07} about
the log-convexity.

Many sequences of binomial coefficients share various unimodality
properties. For example, the sequence
$\left\{\binom{n}{k}\right\}_{k=0}^{n}$ is unimodal and log-concave
in $k$. On the other hand, the sequence
$\left\{\binom{n}{k}\right\}_{n=k}^{+\infty}$ is increasing,
log-concave and convex in $n$ (see Comtet~\cite{Com74} for example).
As usual, let $\binom{n}{k}=0$ unless $0\le k\le n$. Tanny and
Zuker~\cite{TZ74,TZ76} showed the unimodality and log-concavity of
the binomial sequences $\left\{\binom{n_0-i}{i}\right\}_i$ and
$\left\{\binom{n_0-i\D}{i}\right\}_i$. Very recently, Belbachir {\it
et al}~\cite{BB07} showed the unimodality and log-concavity of the
binomial sequence $\left\{\binom{n_0+i}{i\D}\right\}_i$. They
further proposed the following.
\begin{conj}[{\cite[Conjecture 1]{BB07}}]\label{BB-C}
Let $\binom{n}{k}$ be a fixed element of the Pascal triangle crossed
by a ray. The sequence of binomial coefficients located along this
ray is unimodal.
\end{conj}
\begin{center}
\setlength{\unitlength}{1.0mm}
\begin{picture}(180,60)(10,-1.5)

\put(90,58){$\binom{0}{0}$} \put(85,50){$\binom{1}{0}$}
\put(95,50){$\binom{1}{1}$} \put(80,42){$\binom{2}{0}$}
\put(90,42){$\binom{2}{1}$} \put(100,42){$\binom{2}{2}$}
\put(75,34){$\binom{3}{0}$} \put(85,34){$\binom{3}{1}$}
\put(95,34){$\binom{3}{2}$} \put(105,34){$\binom{3}{3}$}
\put(70,26){$\binom{4}{0}$} \put(80,26){$\binom{4}{1}$}
\put(90,26){$\binom{4}{2}$} \put(100,26){$\binom{4}{3}$}
\put(110,26){$\binom{4}{4}$} \put(65,18){$\binom{5}{0}$}
\put(75,18){$\binom{5}{1}$} \put(85,18){$\binom{5}{2}$}
\put(95,18){$\binom{5}{3}$} \put(105,18){$\binom{5}{4}$}
\put(115,18){$\binom{5}{5}$} \put(60,10){$\binom{6}{0}$}
\put(70,10){$\binom{6}{1}$} \put(80,10){$\binom{6}{2}$}
\put(90,10){$\binom{6}{3}$} \put(100,10){$\binom{6}{4}$}
\put(110,10){$\binom{6}{5}$} \put(120,10){$\binom{6}{6}$}
\put(55,2){$\binom{7}{0}$} \put(65,2){$\binom{7}{1}$}
\put(75,2){$\binom{7}{2}$} \put(85,2){$\binom{7}{3}$}
\put(95,2){$\binom{7}{4}$} \put(105,2){$\binom{7}{5}$}
\put(115,2){$\binom{7}{6}$} \put(125,2){$\binom{7}{7}$}

\put(99.5,-3.5){\line(-1,5){0.71}} \put(97,7.5){\line(-1,5){3.2}}
\put(92,31){\line(-1,5){3.3}}
\end{picture}

\bigskip Figure 1: a ray with $\D=3$ and $\d=2$.
\end{center}

The object of this paper is to study the unimodality problem of a
sequence of binomial coefficients located in a ray or a transversal
of the Pascal triangle. Let $\left\{\binom{n_i}{k_i}\right\}_{i\ge
0}$ be such a sequence. Then $\left\{n_i\right\}_{i\ge 0}$ and
$\left\{k_i\right\}_{i\ge 0}$ form two arithmetic sequences (see
Figure~1). Clearly, we may assume that the common difference of
$\left\{n_i\right\}_{i\ge 0}$ is nonnegative (by changing the order
of the sequence). For example, the sequence
$\left\{\binom{n_0-i}{i}\right\}_{i=0}^{\lfloor\frac{n_0}{2}\rfloor}$
coincides with the sequence
$\left\{\binom{n_0-\lfloor\frac{n_0}{2}\rfloor+i}{\lfloor\frac{n_0}{2}\rfloor-i}\right\}_{i=0}^{\lfloor\frac{n_0}{2}\rfloor}$
except for the order. On the other hand, the sequence
$\left\{\binom{n_i}{k_i}\right\}_{i\ge 0}$ is the same as the
sequence $\left\{\binom{n_i}{n_i-k_i}\right\}_{i\ge 0}$ by the
symmetry of the binomial coefficients. So we may assume, without
loss of generality, that the common difference of
$\left\{k_i\right\}_{i\ge 0}$ is nonnegative. Thus it suffices to
consider the unimodality of the sequence
$\{\binom{n_{0}+i\D}{k_0+i\d}\}_{i\ge 0}$ for nonnegative integers
$\D$ and $\d$. The following is the main result of this paper, which
in particular, gives an affirmative answer to Conjecture~\ref{BB-C}.
\begin{thm}\label{main-thm}
Let $n_{0},k_{0},\D,\d$ be four nonnegative integers and $n_{0}\geq
k_{0}$. Define the sequence
\begin{equation*}\label{ai}
    C_i=\binom{n_{0}+i\D}{k_0+i\d},\qquad i=0,1,2,\ldots.
\end{equation*}
Then
\begin{itemize}
\item[\rm (i)]
if $\D=\d>0$ or $\d=0$, the sequence is increasing, convex and
log-concave;
\item[\rm (ii)]
if $\D<\d$, the sequence is log-concave and therefore unimodal;
\item[\rm (iii)]
if $\D>\d>0$, the sequence is increasing, convex, and asymptotically
log-convex (i.e., there exists a nonnegative integer $m$ such that
$C_m,C_{m+1},C_{m+2},\ldots$ is log-convex).
\end{itemize}
\end{thm}

This paper is organized as follows. In the next section, we prove
Theorem~\ref{main-thm}. In Section~3, we present a combinatorial
proof of the log-concavity in Theorem~\ref{main-thm} (ii). In
Section~4, we show more precise results about the asymptotically
log-convexity for certain particular sequences of binomial
coefficients in Theorem~\ref{main-thm} (iii). Finally in Section~5,
we propose some open problems and conjectures.

Throughout this paper we will denote by $\lrf{x}$ and $\lrc{x}$ the
largest integer $\le x$ and the smallest integer $\ge x$
respectively.
\section{The proof of Theorem~\ref{main-thm}}
The following result is folklore and we include a proof of it for
completeness.
\begin{lem}\label{lem-subsequence}
If a sequence $\{a_i\}_{i\ge 0}$ of positive numbers is unimodal
(resp. increasing, decreasing, concave, convex, log-concave,
log-convex), then so is its subsequence $\{a_{n_0+i\D}\}_{i\ge 0}$
for arbitrary fixed nonnegative integers $n_0$ and $\D$.
\end{lem}
\begin{proof}
We only consider the log-concavity case since the others are
similar. Let $\{a_i\}_{i\ge 0}$ be a log-concave sequence of
positive numbers. Then the sequence
$\left\{a_{i-1}/a_i\right\}_{i\ge 0}$ is increasing. Hence
$a_{j-1}/a_j\le a_k/a_{k+1}$ for $1\le j\le k$, i.e.,
$a_{j-1}a_{k+1}\le a_ja_k$. Thus
$$a_{n-d}a_{n+d}\le a_{n-d+1}a_{n+d-1}\le
a_{n-d+2}a_{n+d-2}\le\cdots\le a_{n-1}a_{n+1}\le a_n^2,$$ which
implies that the sequence $\{a_{n_0+i\D}\}_{i\ge 0}$ is log-concave.
\end{proof}

\begin{proof}[The proof of Theorem~\ref{main-thm}]
(i)\quad If $\d=0$, then $C_i=\binom{n_0+i\D}{k_0}$. The sequence
$\binom{i}{k_0}$ is increasing, convex and log-concave in $i$, so is
the sequence $C_i$ by Lemma~\ref{lem-subsequence}. The case $\D=\d$
is similar since $C_i=\binom{n_0+i\D}{n_0-k_0}$.

(ii)\quad To show the log-concavity of $\{C_i\}$ when $\D<\d$, it
suffices to show that
$$\binom{n+\D}{k+\d}\binom{n-\D}{k-\d}\le
{\binom{n}{k}}^2$$ for $n\ge k$. Write
\begin{eqnarray*}
  \binom{n+\D}{k+\d}\binom{n-\D}{k-\d}
   &=& \frac{(n+\D)!(n-\D)!}{(n-k+\D-\d)!(k+\d)!(n-k+\d-\D)!(k-\d)!} \\
   &=& \binom{n+\D}{n-k}\binom{n-\D}{n-k}
       \frac{\binom{n-k}{\d-\D}}{\binom{n-k+\d-\D}{\d-\D}}
       \frac{\binom{k-\D}{\d-\D}}{\binom{k+\d}{\d-\D}}.
\end{eqnarray*}
Now
$\binom{n-k}{\d-\D}\le\binom{n-k+\d-\D}{\d-\D},\binom{k-\D}{\d-\D}\le\binom{k+\d}{\d-\D}$
 and $\binom{n+\D}{n-k}\binom{n-\D}{n-k}\le{\binom{n}{n-k}}^2$ by (i).
Hence
$$\binom{n+\D}{k+\d}\binom{n-\D}{k-\d}\le{\binom{n}{n-k}}^2={\binom{n}{k}}^2,$$
as required.

(iii)\quad Assume that $\D>\d>0$. By Vandermonde's convolution
formula, we have
$$\binom{n+\D}{k+\d}=\sum_{r+s=k+\d}\binom{n}{r}\binom{\D}{s}\ge\binom{n}{k}\binom{\D}{\d}\ge2\binom{n}{k},$$
which implies that $\binom{n+\D}{k+\d}>\binom{n}{k}$ and
$\binom{n+\D}{k+\d}+\binom{n-\D}{k-\d}\ge 2\binom{n}{k}$. Hence the
sequence $\{C_i\}$ is increasing and convex.

It remains to show that the sequence $\{C_i\}$ is asymptotically
log-convex. Denote
$$\Delta(i):=\binom{n_0+(i+1)\D}{k_0+(i+1)\d}
\binom{n_0+(i-1)\D}{k_0+(i-1)\d}-{\binom{n_0+i\D}{k_0+i\d}}^2.$$
Then we need to show that $\Delta(i)$ is positive for all
sufficiently large $i$. Write
\begin{eqnarray*}
\Delta(i) &=& \frac{(n_0+i\D)![n_0+(i-1)\D]!}
       {(k_0+i\d)![k_0+(i+1)\d]!
          [n_0-k_0+i(\D-\d)]![n_0-k_0+(i+1)(\D-\d)]!}\times\\
&&\times
 \left\{\prod_{j=1}^{\D}(n_0+i\D+j)
  \prod_{j=1}^{\D-\d}[n_0-k_0+(i-1)(\D-\d)+j]
   \prod_{j=1}^{\d}[k_0+(i-1)\d+j]\right.\\
&&\left.-\prod_{j=1}^{\D}[n_0+(i-1)\D+j]
        \prod_{j=1}^{\D-\d}[n_0-k_0+i(\D-\d)+j]
        \prod_{j=1}^{\d}(k_0+i\d+j)\right\}\\
&=& \frac{(n_0+i\D)![n_0+(i-1)\D]!\D^{\D} \d^{\d}(\D-\d)^{(\D-\d)}}
       {(k_0+i\d)![k_0+(i+1)\d]!
          [n_0-k_0+i(\D-\d)]![n_0-k_0+(i+1)(\D-\d)]!}P(i),
\end{eqnarray*}
where
\begin{eqnarray*}
P(i) &=&
   \prod_{j=1}^{\D}\left(i+\frac{n_0+j}{\D}\right)
   \prod_{j=1}^{\D-\d}\left(i+\frac{n_0-k_0-\D+\d+j}{\D-\d}\right)
 \prod_{j=1}^{\d}\left(i+\frac{k_0-\d+j}{\d}\right)\\
&& -\prod_{j=1}^{\D}\left(i+\frac{n_0-\D+j}{\D}\right)
   \prod_{j=1}^{\D-\d}\left(i+\frac{n_0-k_0+j}{\D-\d}\right)
   \prod_{j=1}^{\d}\left(i+\frac{k_0+j}{\d}\right).
\end{eqnarray*}
Then it suffices to show that $P(i)$ is positive for sufficiently
large $i$. Clearly, $P(i)$ can be viewed as a polynomial in $i$. So
it suffices to show that the leading coefficient of $P(i)$ is
positive.

Note that $P(i)$ is the difference of two monic polynomials of
degree $2\D$. Hence its degree is less than $2\D$. Denote
$$P(i)=a_{2\D-1}i^{2\D-1}+a_{2\D-2}i^{2\D-2}+\cdots.$$
By Vieta's formula, we have
\begin{eqnarray*}
a_{2\D-1} &=&
-\left(\sum_{j=1}^{\D}\frac{n_0+j}{\D}+\sum_{j=1}^{\D-\d}\frac{n_0-k_0-\D+\d+j}{\D-\d}+\sum_{j=1}^{\d}\frac{k_0-\d+j}{\d}\right)\\
& & +\left(\sum_{j=1}^{\D}\frac{n_0-\D+j}{\D}+\sum_{j=1}^{\D-\d}\frac{n_0-k_0+j}{\D-\d}+\sum_{j=1}^{\d}\frac{k_0+j}{\d}\right)\\
&=& \sum_{j=1}^{\D}\left(\frac{n_0-\D+j}{\D}-\frac{n_0+j}{\D}\right)\\
&& +\sum_{j=1}^{\D-\d}
{\left(\frac{n_0-k_0+j}{\D-\d}-\frac{n_0-k_0-\D+\d+j}{\D-\d}
\right)}\\
&& +\sum_{j=1}^{\d}{\left(\frac{k_0+j}{\d}-\frac{k_0-\d+j}{\d}
\right)}\\
&=& \sum_{j=1}^{\D}(-1)+\sum_{j=1}^{\D-\d}1+\sum_{j=1}^{\d}1\\
&=& -\D+(\D-\d)+\d\\
&=& 0.
\end{eqnarray*}

Using the identity
$$\sum_{1\le i<j\le n} x_ix_j=\frac{1}{2}\left[\left(\sum_{i=1}^{n}
x_i\right)^2-\sum_{i=1}^{n} x_i^2\right],$$ we obtain again by
Vieta's formula
\begin{eqnarray*}
a_{2\D-2} &=&
\frac{1}{2}\left[\left(\sum_{j=1}^{\D}\frac{n_0+j}{\D}+\sum_{j=1}^{\D-\d}\frac{n_0-k_0-\D+\d+j}{\D-\d}+\sum_{j=1}^{\d}\frac{k_0-\d+j}{\d}\right)^2\right.\\
&&\left.-\left(\sum_{j=1}^{\D}\frac{n_0+j}{\D}\right)^2-\left(\sum_{j=1}^{\D-\d}\frac{n_0-k_0-\D+\d+j}{\D-\d}\right)^2-\left(\sum_{j=1}^{\d}\frac{k_0-\d+j}{\d}\right)^2\right]\\
&&-\frac{1}{2}\left[\left(\sum_{j=1}^{\D}\frac{n_0-\D+j}{\D}+\sum_{j=1}^{\D-\d}\frac{n_0-k_0+j}{\D-\d}+\sum_{j=1}^{\d}\frac{k_0+j}{\d}\right)^2\right.\\
&&\left.-\left(\sum_{j=1}^{\D}\frac{n_0-\D+j}{\D}\right)^2-\left(\sum_{j=1}^{\D-\d}\frac{n_0-k_0+j}{\D-\d}\right)^2-\left(\sum_{j=1}^{\d}\frac{k_0+j}{\D-\d}\right)^2\right].
\end{eqnarray*}
But $a_{2\D-1}=0$ implies
\begin{eqnarray*}
&&\left(\sum_{j=1}^{\D}\frac{n_0+j}{\D}+\sum_{j=1}^{\D-\d}\frac{n_0-k_0-\D+\d+j}{\D-\d}+\sum_{j=1}^{\d}\frac{k_0-\d+j}{\d}\right)^2\\
&=&\left(\sum_{j=1}^{\D}\frac{n_0-\D+j}{\D}+\sum_{j=1}^{\D-\d}\frac{n_0-k_0+j}{\D-\d}+\sum_{j=1}^{\d}\frac{k_0+j}{\d}\right)^2,
\end{eqnarray*} so we have
\begin{eqnarray*}
a_{2\D-2} &=&
\frac{1}{2}\sum_{j=1}^{\D}\left[\left(\frac{n_0-\D+j}{\D}
\right)^2-\left(\frac{n_0+j}{\D}\right)^2\right]\\
&& +\frac{1}{2}\sum_{j=1}^{\D-\d}
\left[\left(\frac{n_0-k_0+j}{\D-\d}
\right)^2-\left(\frac{n_0-k_0-\D+\d+j}{\D-\d}\right)^2\right]\\
&& +\frac{1}{2}\sum_{j=1}^{\d}\left[\left(\frac{k_0+j}{\d}
\right)^2-\left(\frac{k_0-\d+j}{\d}\right)^2\right]\\
&=& -\frac{1}{2}\sum_{j=1}^{\D}\frac{2n_0-\D+2j}{\D}
+\frac{1}{2}\sum_{j=1}^{\D-\d}\frac{2(n_0-k_0)-(\D-\d)+2j}{\D-\d}
+\frac{1}{2}\sum_{j=1}^{\d}\frac{2k_0-\d+2j}{\d}\\
&=& -\frac{1}{2}(2n_0+1)+\frac{1}{2}(2n_0-2k_0+1)+\frac{1}{2}(2k_0+1)\\
&=& {\frac{1}{2}}.
\end{eqnarray*}
Thus $P(i)$ is a polynomial of degree $2\D-2$ with positive leading
coefficient, as desired.  This completes the proof of the theorem.
\end{proof}
\section{Combinatorial proof of the log-concavity}
In Section~2 we have investigated the unimodality of sequences of
binomial coefficients by an algebraic approach. It is natural to ask
for a combinatorial interpretation. Lattice path techniques have
been shown to be useful in solving the unimodality problem. As an
example, we present a combinatorial proof of Theorem~\ref{main-thm}
(ii) following B\'ona and Sagan's technique in \cite{BS03}.

Let $\mathbb{Z}^2=\{(x,y):x,y\in\mathbb{Z}\}$ denote the
two-dimensional integer lattice. A lattice path is a sequence
$P_1,P_2,\ldots,P_{\ell}$ of lattice points on $\mathbb{Z}^2$. A
southeastern lattice path is a lattice path in which each step goes
one unit to the south or to the east. Denote by $P(n,k)$ the set of
southeastern lattice paths from the point $(0,n-k)$ to the point
$(k,0)$. Clearly, the number of such paths is the binomial
coefficient $\binom{n}{k}$.

Recall that, to show the log-concavity of
$C_i=\binom{n_0+i\D}{k_0+i\d}$ where $n_0\ge k_0$ and $\D<\d$, it
suffices to show $\binom{n+\D}{k+\d}\binom{n-\D}{k-\d}\le
{\binom{n}{k}}^2$ for $n\ge k$. Here we do this by constructing an
injection
$$\phi:P(n+\D,k+\d)\times
P(n-\D,k-\d)\longrightarrow P(n,k)\times P(n,k).$$

Consider a path pair $(p,q)\in P(n+\D,k+\d)\times P(n-\D,k-\d)$.
Then $p$ and $q$ must intersect. Let $I_1$ be the first
intersection. For two points $P(a,b)$ and $Q(a,c)$ with the same
$x$-coordinate, define their vertical distance to be $d_v(P,Q)=b-c$.
Then the vertical distance from a point of $p$ to a point of $q$
starts at $2(\d-\D)$ for their initial points and ends at $0$ for
their intersection $I_1$. Thus there must be a pair of points $P\in
p$ and $Q\in q$ before $I_1$ with $d_v(P,Q)=\d-\D$. Let $(P_1,Q_1)$
be the first such pair of points. Similarly, after the last
intersection $I_2$ there must be a last pair of points $P_2\in p$
and $Q_2\in q$ with the horizontal distance $d_h(P_2,Q_2)=-\d$ (the
definition of $d_h$ is analogous to that of $d_v$). Now $p$ is
divided by two points $P_1,P_2$ into three subpaths $p_1,p_2,p_3$
and $q$ is divided by $Q_1,Q_2$ into three subpaths $q_1,q_2,q_3$.
Let $p'_1$ be obtained by moving $p_1$ down to $Q_1$ south $\d-\D$
units and $p'_3$ be obtained by moving $p_3$ right to $Q_2$ east
$\d$ units. Then we obtain a southeastern lattice path $p'_1q_2p'_3$
in $P(n,k)$. We can similarly obtain the second southeastern lattice
path $q'_1p_2q'_3$ in $P(n,k)$, where $q'_1$ is $q_1$ moved north
$\d-\D$ units and $q'_3$ is $q_3$ moved west $\d$ units. Define
$\phi(p,q)=(p'_1q_2p'_3,q'_1p_2q'_3)$. It is not difficult to verify
that $\phi$ is the required injective. We omit the proof for
brevity.
\begin{center}
\begin{picture}(400,110)
\setlength{\unitlength}{0.70mm} \put(0,0){\begin{picture}(50,50)
 \multiput(0,0)(10,0){9}{\circle*{1.1}}
 \multiput(0,10)(10,0){9}{\circle*{1.1}}
 \multiput(0,20)(10,0){9}{\circle*{1.1}}
 \multiput(0,30)(10,0){9}{\circle*{1.1}}
 \multiput(0,40)(10,0){9}{\circle*{1.1}}
 \multiput(0,50)(10,0){9}{\circle*{1.1}}

 \linethickness{1.2pt} \put(0,50){\line(1,0){20}}
 \put(20,30){\line(1,0){10}} \put(30,10){\line(1,0){10}}
 \put(20,30){\line(0,1){20}} \put(30,10){\line(0,1){20}}
 \put(40,0){\line(0,1){10}} \thinlines

 \linethickness{1.2pt}
 \multiput(10,20)(0,2){5}{\line(0,1){0.8}}
 \multiput(50,10)(0,2){5}{\line(0,1){0.8}}
 \multiput(80,0)(0,2){5}{\line(0,1){0.8}}
 \thinlines

 \linethickness{1.2pt}
 \multiput(50,10)(2,0){15}{\line(1,0){0.8}}
 \multiput(10,20)(2,0){20}{\line(1,0){0.8}}
 \multiput(0,30)(2,0){5}{\line(1,0){0.8}}
 \thinlines

 \put(30,20){\circle*{2.0}} \put(20,30){\circle*{2.0}}
 \put(20,20){\circle*{2.0}} \put(40,10){\circle*{2.0}}
 \put(60,10){\circle*{2.0}}

\footnotesize{
 \put(21,44){$p_1$} \put(11,24){$q_1$}
 \put(45.5,22){$q_2$}  \put(74,12){$q_3$}  \put(31,24){$p_2$}
 \put(41,4){$p_3$}

 \put(15,26){$P_1$} \put(14.5,15.5){$Q_1$} \put(25,15){$I_1$}
 \put(34.5,5.5){$P_2$} \put(54,5.5){$Q_2$}}
 \end{picture}}

\put(16,0){\begin{picture}(50,50)
 \multiput(104,0)(10,0){9}{\circle*{1.1}}
 \multiput(104,10)(10,0){9}{\circle*{1.1}}
 \multiput(104,20)(10,0){9}{\circle*{1.1}}
 \multiput(104,30)(10,0){9}{\circle*{1.1}}
 \multiput(104,40)(10,0){9}{\circle*{1.1}}
 \multiput(104,50)(10,0){9}{\circle*{1.1}}

 \linethickness{1.2pt} \put(104,40){\line(1,0){20}}
 \put(124,20){\line(0,1){20}} \put(124,20){\line(1,0){30}}
 \put(154,10){\line(0,1){10}} \put(154,10){\line(1,0){10}}
 \put(164,0){\line(0,1){10}}
 \thinlines

 \linethickness{1.2pt}
 \multiput(114,30)(0,2){5}{\line(0,1){0.8}}
 \multiput(134,10)(0,2){10}{\line(0,1){0.8}}
 \multiput(154,10)(0,2){5}{\line(0,1){0.8}}
 \multiput(163.5,0)(0,2){5}{\line(0,1){0.8}}
 \multiput(114,30)(2,0){10}{\line(1,0){0.8}}
 \multiput(104,39.5)(2,0){5}{\line(1,0){0.8}}
 \multiput(134,9.5)(2,0){15}{\line(1,0){0.8}}
 \thinlines

 \put(134,20){\circle*{2.0}} \put(124,30){\circle*{2.0}}
 \put(124,20){\circle*{2.0}} \put(144,10){\circle*{2.0}}
 \put(164,10){\circle*{2.0}}
 \footnotesize{
 \put(135,24){$p_2$} \put(149.5,22){$q_2$} \put(125,35){$p'_1$}
 \put(165,4){$p'_3$} \put(115,35){$q'_1$} \put(148,11.5){$q'_3$}

 \put(118,25){$P_1$}
 \put(118,15.5){$Q_1$} \put(129,15){$I_1$} \put(138.5,5.5){$P_2$}
 \put(157.5,5.5){$Q_2$}}

 \end{picture}}
\end{picture}

\bigskip
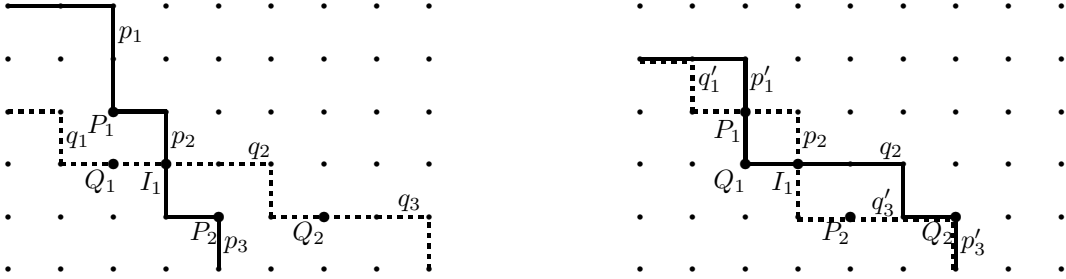
Figure 2: the constructing of $\phi$.
\end{center}
\section{Asymptotic behavior of the log-convexity}
Theorem~\ref{main-thm} (iii) tells us that the sequence
$C_i=\binom{n_0+i\D}{k_0+i\d}$ is asymptotically log-convex when
$\D>\d>0$. We can say more for a certain particular sequence of
binomial coefficients. For example, it is easy to verify that the
central binomial coefficients $\binom{2i}{i}$ is log-convex for
$i\ge 0$ (see Liu and Wang~\cite{LW07} for a proof). In this section
we give two generalizations of this result. The first one is that
every sequence of binomial coefficients located along a ray with
origin $\binom{0}{0}$ is log-convex.
\begin{prop}\label{ray-0}
Let $\D$ and $\d$ be two positive integers and $\D>\d>0$. Then the
sequence $\left\{\binom{i\D}{i\d}\right\}_{i\ge 0}$ is log-convex.
\end{prop}

Before showing Proposition~\ref{ray-0}, we first demonstrate two
simple but useful facts.

Let $\alpha=(a_1,a_2,\ldots,a_n)$ and $\beta=(b_1,b_2,\ldots,b_n)$
be two $n$-tuples of real numbers. We say that $\alpha$ {\it
alternates left of} $\beta$, denoted by $\alpha\preceq\beta$, if
$$a_1^*\le b_1^*\le a_2^*\le b_2^*\cdots\le a_n^*\le b_n^*,$$
where $a_j^*$ and $b_j^*$ are the $j$th smallest elements of
$\alpha$ and $\beta$, respectively.
\begin{description}
  \item[Fact 1] Let $f(x)$ be a nondecreasing function.
    If $(a_1,a_2,\dots,a_n)\preceq(b_1,b_2,\ldots,b_n)$,
    then $\prod_{i=1}^{n}f(a_i) \le \prod_{i=1}^{n}f(b_i)$.
  \item[Fact 2] Let $x_1,x_2,y_1,y_2$ be four positive numbers and
    $\frac{x_1}{y_1}\le\frac{x_2}{y_2}$. Then
    $\frac{x_1}{y_1}\le\frac{x_1+x_2}{y_1+y_2}\le\frac{x_2}{y_2}$.
\end{description}
\begin{proof}[Proof of Proposition~\ref{ray-0}]
By Lemma~\ref{lem-subsequence}, we may assume, without loss of
generality, that $\D$ and $\d$ are coprime. We need to show that
$$\Delta(i):=\binom{(i+1)\D}{(i+1)\d}\binom{(i-1)\D}{(i-1)\d}-{\binom{i\D}{i\d}}^2\ge 0$$
for all $i\ge 1$. Write
\begin{eqnarray*}
\Delta(i)=\frac{(i\D)![(i-1)\D]!\D^{\D}
\d^{\d}(\D-\d)^{(\D-\d)}\prod_{j=1}^{\D}\left(i+\frac{j}{\D}\right)
   \prod_{j=1}^{\d}\left(i+\frac{j}{\d}\right)
   \prod_{j=1}^{\D-\d}\left(i+\frac{j}{\D-\d}\right)}
   {(i\d)![(i+1)\d]![i(\D-\d)]![(i+1)(\D-\d)]!}Q(i),
\end{eqnarray*}
where
$$Q(i)=\prod_{j=1}^{\d}\left(1-\frac{1}{i+\frac{j}{\d}}\right)
       \prod_{j=1}^{\D-\d}\left(1-\frac{1}{i+\frac{j}{\D-\d}}\right)
      -\prod_{j=1}^{\D}\left(1-\frac{1}{i+\frac{j}{\D}}\right).$$
Then we only need to show that $Q(i)\ge 0$ for $i\ge 1$. We do this
by showing
$$\left(\frac{1}{\D},\ldots,\frac{\D-1}{\D},\frac{\D}{\D}\right)
\preceq\left(\frac{1}{\d},\ldots,\frac{\d-1}{\d},\frac{\d}{\d},
\frac{1}{\D-\d},\ldots,\frac{\D-\d-1}{\D-\d},\frac{\D-\d}{\D-\d}\right),$$
or equivalently,
$$\left(\frac{1}{\D},\ldots,\frac{\D-1}{\D}\right)
\preceq\left(\frac{1}{\d},\ldots,\frac{\d-1}{\d},
\frac{1}{\D-\d},\ldots,\frac{\D-\d-1}{\D-\d},1\right).$$ Note that
$(\D,\d)=1$ implies all fractions
$\left\{\frac{j}{\D}\right\}_{j=1}^{\D-1}$,
$\left\{\frac{j}{\d}\right\}_{j=1}^{\d-1}$ and
$\left\{\frac{j}{\D-\d}\right\}_{j=1}^{\D-\d-1}$ are different.
Hence it suffices to show that every term of
$\left\{\frac{j}{\d}\right\}_{j=1}^{\d-1}\bigcup\left\{\frac{j}{\D-\d}\right\}_{j=1}^{\D-\d-1}$
is precisely in one of $\D-2$ open intervals
$\left(\frac{k}{\D},\frac{k+1}{\D}\right)$, where $k=1,\ldots,\D-2$.
Indeed, neither two terms of
$\left\{\frac{j}{\d}\right\}_{j=1}^{\d-1}$ nor two terms of
$\left\{\frac{j}{\D-\d}\right\}_{j=1}^{\D-\d-1}$ are in the same
interval since their difference is larger than $\frac{1}{\D}$. On
the other hand, if $\frac{j}{\d}$ and $\frac{j'}{\D-\d}$ are in a
certain interval $\left(\frac{k}{\D},\frac{k+1}{\D}\right)$, then so
is $\frac{j+j'}{\D}$ by Fact~2, which is impossible. Thus there
exists precisely one term of
$\left\{\frac{j}{\d}\right\}_{j=1}^{\d-1}\bigcup\left\{\frac{j}{\D-\d}\right\}_{j=1}^{\D-\d-1}$
in every open interval $\left(\frac{k}{\D},\frac{k+1}{\D}\right)$,
as desired. This completes our proof.
\end{proof}

For the second generalization of the log-convexity of the central
binomial coefficients, we consider sequences of binomial
coefficients located along a vertical ray with origin
$\binom{n_0}{0}$ in the Pascal triangle.
\begin{prop}\label{v-ray}
Let $n_0\ge 0$ and $V_i(n_0)=\binom{n_0+2i}{i}$. Then
$V_0(n_0),V_1(n_0),\ldots,V_{m}(n_0)$ is log-concave and
$V_{m-1}(n_0),V_{m}(n_0),V_{m+1}(n_0),\ldots$ is log-convex, where
$m=n_0^2-\lrc{\frac{n_0}{2}}$.
\end{prop}
\begin{proof}
The sequence $V_i(0)=\binom{2i}{i}$ is just the central binomial
coefficients and therefore log-convex for $i\ge 0$. It implies that
the sequence $V_i(1)=\binom{1+2i}{i}$ is log-convex for $i\ge 0$
since $V_i(1)=\frac{1}{2}V_{i+1}(0)$. Now let $n_0\ge 2$ and define
$f(i)=V_{i+1}(n_0)/V_i(n_0)$ for $i\ge 0$. Then, to show the
statement, it suffices to show that
\begin{equation}\label{f-v}
f(0)>f(1)>\cdots >f(m-1)\quad\text{and}\quad
f(m-1)<f(m)<f(m+1)<\cdots\end{equation} for
$m=n_0^2-\lrc{\frac{n_0}{2}}$.

By the definition we have
\begin{equation}\label{fi}
f(i)=\frac{\binom{n_0+2(i+1)}{i+1}}{\binom{n_0+2i}{i}}=\frac{(n_0+2i+1)(n_0+2i+2)}{(i+1)(n_0+i+1)}.
\end{equation} The derivative of $f(i)$ with respect to $i$ is
$$f'(i)=\frac{2i^2-2(n_0-2)(n_0+1)i-(n_0+1)(n_0^2-2)}{(i+1)^2(n_0+i+1)^2}.$$
The numerator of $f'(i)$ has the unique positive zero
\begin{eqnarray*}
r &=&\frac{2(n_0-2)(n_0+1)+\sqrt{4(n_0-2)^2(n_0+1)^2+8(n_0+1)(n_0^2-2)}}{4}\\
  &=& \frac{(n_0-2)(n_0+1)}{2}+\frac{n_0\sqrt{n_0^2-1}}{2}.
\end{eqnarray*}
It implies that $f'(i)<0$ for $0\le i<r$ and $f'(i)>0$ for $i>r$.
Thus we have
\begin{equation}\label{f-p}
f(0)>f(1)>\cdots>f(\lrf{r})\quad\text{and}\quad f(\lrc{r})<
f(\lrc{r}+1)<f(\lrc{r}+2)<\cdots.\end{equation} It remains to
compare the values of $f(\lrf{r})$ and $f(\lrc{r})$. Note that
$$\frac{n_0^2-n_0\sqrt{n_0^2-1}}{2}=\frac{n_0}{2(n_0+\sqrt{n_0^2-1})}<\frac{1}{2}.$$
Hence $$\lrc{\frac{n_0\sqrt{n_0^2-1}}{2}}=\left\{\begin{array}{ll}
 \frac{n_0^2}{2}, & \hbox{if $n_0$ is even;} \\
 \frac{n_0^2+1}{2}, & \hbox{if $n_0$ is odd,}
 \end{array}
 \right.$$
and so
$$\lrc{r}=\frac{(n_0-2)(n_0+1)}{2}+\lrc{\frac{n_0\sqrt{n_0^2-1}}{2}}
=\left\{
   \begin{array}{ll}
     n_0^2-\frac{n_0}{2}-1, & \hbox{if $n_0$ is even;} \\
     n_0^2-\frac{n_0+1}{2}, & \hbox{if $n_0$ is odd.}
   \end{array}
 \right.
$$

If $n_0$ is even, then by (\ref{fi}) we have
$$f(\lrc{r})=\frac{16n_0^2-8}{4n_0^2-1}=4-\frac{4}{4n_0^2-1}$$ and
$$f(\lrf{r})=f(\lrc{r}-1)=\frac{16n_0^4-40n_0^2+16}{4n_0^4-9n_0^2+4}=4-\frac{4(n_0^2-2)}{4n_0^4-9n_0^2+4}.$$
Thus $f(\lrf{r})>f(\lrc{r})$ since
$f(\lrf{r})-f(\lrc{r})=\frac{8}{(4n_0^2-1)(4n_0^4-9n_0^2+4)}>0$.
Also, $\lrc{r}=m-1$. Combining (\ref{f-p}) we obtain (\ref{f-v}).

If $n_0$ is odd, then
$$f(\lrc{r})=4-\frac{4(n_0^2+1)}{4n_0^4+3n_0^2+1}$$ and
$$f(\lrf{r})=4-\frac{4(n_0^2-1)}{4n_0^4-5n_0^2+1}.$$
It is easy to verify that $f(\lrf{r})<f(\lrc{r})$. Also,
$\lrf{r}=\lrc{r}-1=m-1$. Thus (\ref{f-v}) follows. This completes
our proof.
\end{proof}
\section{Concluding remarks and open problems}
\hspace*{\parindent}
In this paper we show that the sequence
$C_i=\binom{n_0+i\D}{k_0+i\d}$ is unimodal when $\D<\d$. A further
problem is to find out the value of $i$ for which $C_i$ is a
maximum. Tanny and Zuker~\cite{TZ74,TZ76,TZ78} considered such a
problem for the sequence $\binom{n_0-i\D}{i}$. For example, it is
shown that the sequence $\binom{n_0-i}{i}$ attains the maximum when
$i=\lrf{(5n_0+7-\sqrt{5n_0^2+10n_0+9})/10}$. Let $r(n_0,d)$ be the
least integer at which $\binom{n_0-i\D}{i}$ attains its maximum.
They investigated the asymptotic behavior of $r(n_0,d)$ for
$d\rightarrow\infty$ and concluded with a variety of unsolved
problems concerning the numbers $r(n_0,d)$. An interesting problem
is to consider analogue for the general binomial sequence
$C_i=\binom{n_0+i\D}{k_0+i\d}$ when $\D<\d$. It often occurs that
unimodality of a sequence is known, yet to determine the exact
number and location of modes is a much more difficult task.

A finite sequence of positive numbers $a_0,a_1,\ldots,a_n$ is called
a {\it P\'olya frequency sequence} if its generating function
$P(x)=\sum_{i=0}^{n}a_ix^i$ has only real zeros. By the Newton's
inequality, if $a_0,a_1,\ldots,a_n$ is a P\'olya frequency sequence,
then
$$a_i^2\ge a_{i-1}a_{i+1}\left(1+\frac{1}{i}\right)\left(1+\frac{1}{n-i}\right)$$
for $1\le i\le n-1$, and the sequence is therefore log-concave and
unimodal with at most two modes (see Hardy, Littlewood and
P\'olya~\cite[p. 104]{HLP52}). Darroch~\cite{Dar64} further showed
that each mode $m$ of the sequence $a_0,a_1,\ldots,a_n$ satisfies
$$\left\lfloor\frac{P'(1)}{P(1)}\right\rfloor\le m\le \left\lceil\frac{P'(1)}{P(1)}\right\rceil.$$
We refer the reader to
\cite{Bre89,Bre94,Erd52,LWaam07,MWejc08,Sta89,WYjcta05} for more
information.

For example, the binomial coefficients
$\binom{n}{0},\binom{n}{1},\ldots,\binom{n}{n}$ is a P\'olya
frequency sequence with the unique mode $n/2$ for even $n$ and two
modes $(n\pm 1)/2$ for odd $n$. On the other hand, the sequence
$\binom{n}{0},\binom{n-1}{1},\binom{n-2}{2},\ldots,\binom{\lrc{n/2}}{\lrf{n/2}}$
is a P\'olya frequency sequence since its generating function is
precisely the matching polynomial of a path on $n$ vertices. Hence
we make the more general conjecture that every sequence of binomial
coefficients located in a transversal of the Pascal triangle is a
P\'olya frequency sequence.
\begin{conj}
Let $C_i=\binom{n_0+i\D}{k_0+i\d}$ where $n_0\ge k_0$ and $\d>\D>
0$. Then the finite sequence $\{C_i\}_i$ is a P\'olya frequency
sequence.
\end{conj}

In Proposition~\ref{v-ray} we have shown that the sequence
$V_i(n_0)=\binom{n_0+2i}{i}$ is first log-concave and then
log-convex. It is possible that an arbitrary sequence of binomial
coefficients located along a ray in the Pascal triangle has the same
property as the sequence $V_i(n_0)$. We leave this as a conjecture
to end this paper.
\begin{conj}
Let $C_i=\binom{n_0+i\D}{k_0+i\d}$ where $n_0\ge k_0$ and $\D>\d>0$.
Then there is a nonnegative integer $m$ such that
$C_0,C_1,\ldots,C_{m-1},C_m$ is log-concave and
$C_{m-1},C_m,C_{m+1},\ldots$ is log-convex.
\end{conj}
\section*{Acknowledgments}

The authors thank Feng Guo, Po-Yi Huang and Yeong-Nan Yeh for
helpful discussions.



\begin{thebibliography}{99}
\bibitem{BB07}
H. Belbachir, F. Bencherif and L. Szalay, Unimodality of certain
sequences connected with binomial coefficients, J. Integer Seq. 10
(2007), Article 07. 2. 3.
\bibitem{BS03}
M. B\'ona and B. Sagan, Two injective proofs of a conjecture of
Simion, J. Combin. Theory Ser. A 10 (2003) 79--89.
\bibitem{Bre89}
F. Brenti, Unimodal, log-concave, and P\'olya frequency sequences in
combinatorics,  Mem. Amer. Math. Soc. 413 (1989).
\bibitem{Bre94}
F. Brenti, Log-concave and unimodal sequences in algebra,
combinatorics, and geometry: An update, Contemp. Math. 178 (1994)
71--89.
\bibitem{Com74}
L. Comtet, Advanced Combinatorics, Reidel, Dordrecht, 1974.
\bibitem{Dar64}
J. N. Darroch, On the distribution of the number of successes in
independent trials, Ann. Math. Statist. 35 (1964) 1317--1321.
\bibitem{DV08}
T. Do\v sli\'c and D. Veljan, Logarithmic behavior of some
combinatorial sequences, Discrete Math. 308 (2008) 2182--2212.
\bibitem{Erd52}
P. Erd\H{o}s, On a conjecture of Hammersley, J. London Math. Soc. 28
(1953) 232--236.
\bibitem{HLP52}
G. H. Hardy, J. E. Littlewood and G. P\'olya, Inequalities,
Cambridge University Press, Cambridge, 1952.
\bibitem{LW07}
L. L. Liu and Y. Wang, On the log-convexity of combinatorial
sequences, Adv. in. Appl. Math.  39 (2007) 453--476.
\bibitem{LWaam07}
L. L. Liu and Y. Wang, A unified approach to polynomial sequences
with only real zeros, Adv. in Appl. Math. 38 (2007) 542--560.
\bibitem{MWejc08}
S.-M. Ma and Y. Wang, $q$-Eulerian polynomials and polynomials with
only real zeros, Electron. J. Combin. 15 (2008), Research Paper 17,
9 pp.
\bibitem{Sta89}
R. P. Stanley, Log-concave and unimodal sequences in algebra,
combinatorics and geometry, Ann. New York Acad. Sci. 576 (1989)
500--534.
\bibitem{TZ74}
S. Tanny and M. Zuker, On a unimodal sequence of binomial
coefficients, Discrete Math. 9 (1974) 79--89.
\bibitem{TZ76}
S. Tanny and M. Zuker, On a unimodal sequence of binomial
coefficients II, J. Combin. Inform. System Sci. 1 (1976) 81--91.
\bibitem{TZ78}
S. Tanny and M. Zuker, Analytic methods applied to a sequence of
binomial coefficients, Discrete Math. 24 (1978) 299--310.
\bibitem{Wjcta02}
Y. Wang, A simple proof of a conjecture of Simion, J. Combin. Theory
Ser. A 100 (2002) 399--402.
\bibitem{Weujc02}
Y. Wang, Proof of a conjecture of Ehrenborg and Steingr¨ªmsson on
excedance statistic, European J. Combin. 23 (2002) 355--365.
\bibitem{Wlaa03}
Y. Wang, Linear transformations preserving log-concavity, Linear
Algebra Appl. 359 (2003) 161--167.
\bibitem{WYjcta05}
Y. Wang and Y.-N. Yeh, Polynomials with real zeros and P\'olya
frequency sequences, J. Combin. Theory Ser. A 109 (2005) 63--74.
\bibitem{WYeujc05}
Y. Wang and Y.-N. Yeh, Proof of a conjecture on unimodality,
European J. Combin. 26 (2005) 617--627.
\bibitem{WYjcta07}
Y. Wang and Y.-N. Yeh, Log-concavity and LC-positivity, J. Combin.
Theory Ser. A 114 (2007) 195--210.
\end{thebibliography}
\end{document}